\documentclass[reqno, a4paper, 12pt]{amsart}
\usepackage{amsmath, amssymb,amsthm}
\pdfoutput=1
\usepackage{color}

\usepackage{xcolor}

\usepackage{txfonts}
\usepackage{enumerate}
\usepackage{caption}
\usepackage{accents}

\newtheorem{theorem}{Theorem}

\newtheorem{pseudotheorem}[theorem]{Formal Statement}

\newtheorem{lemma}[theorem]{Lemma}

{Important Convention}

\theoremstyle{remark}

  \newcommand{\F}{\mathcal{F}}

 \newcommand{\PP}{\mathbb{P}}

 \renewcommand{\phi}{\varphi}

\renewcommand{\P}{\mathbb{P}}

\newcommand{\Q}{\mathbb{Q}}
\newcommand{\R}{\mathbb{R}}

\newcommand{\de}{\mathrm{d}}

\DeclareMathOperator{\supp}{supp}
\newcommand{\bes}{\begin{subequations}}
\newcommand{\ees}{\end{subequations}}
\newcommand{\eea}{\end{eqnarray}}

\newcommand{\cF}{\mathcal{F}}

\newcommand{\NN}{{\mathbb N}}

\newcommand{\cal}{\mathcal}

\newcommand{\QQ}{{\mathbb Q}}

\usepackage{bbm}

\renewcommand{\epsilon}{\varepsilon}

\DeclareMathOperator{\proj}{proj}

\newcommand{\fourIdx}[5]{%
\setbox1=\hbox{\ensuremath{^{#1}}}%
 \setbox2=\hbox{\ensuremath{_{#2}}}%
 \setbox5=\hbox{\ensuremath{#5}}%
 \hspace{\ifnum\wd1>\wd2\wd1\else\wd2\fi}%
 \ensuremath{\copy5^{\hspace{-\wd1}\hspace{-\wd5}#1\hspace{\wd5}#3}%
 _{\hspace{-\wd2}\hspace{-\wd5}#2\hspace{\wd5}#4}%
 }}

\numberwithin{equation}{section}
\numberwithin{theorem}{section}

\renewcommand{\subset}{\subseteq}

\usepackage{subcaption}
\usepackage{graphicx}
\usepackage{todonotes}

\newcommand{\cpl}{\text{cpl}}

\author{Julio Backhoff-Veraguas}
\author{Mathias Beiglb\"ock}
\author{Giovanni Conforti}

\title{A non-linear monotonicity principle and applications to Schr\"odinger-type problems}
\thanks{JB and MB acknowledge the Austrian Science Fund (FWF) for its support via the project Y00782. GC acknowledges funding from the grant SPOT (ANR-20-CE40-0014). }

\begin{document}
\maketitle

\begin{abstract}
A basic idea in optimal transport is that optimizers can be characterized through a geometric property of their support sets called \emph{cyclical monotonicity}. In recent years, similar \emph{monotonicity principles} have found applications in other fields where infinite-dimensional linear optimization problems play an important role. 

In this note, we observe how this approach can be transferred to \emph{non-linear} optimization problems. Specifically we establish a monotonicity principle that is applicable  to the Schr\"odinger problem and use it to characterize the structure of optimizers for target functionals beyond relative entropy. In contrast to classical convex duality approaches, a main novelty is that the monotonicity principle allows to deal also \emph{with non-convex functionals.}

%
 
\medskip \emph{keywords:} cyclical monotonicity,  monotonicity principle, Schr\"odinger problem, 
$L^2$ divergence,  non-linear optimization, 
\end{abstract}

\section{Introduction and  main results}

 \subsection{Motivation from optimal transport}
Given probabilities $\mu$ and $\nu$ on Polish 
spaces $X$ and  $Y$, and a cost function $c:X\times Y\to \R_+$, 
the Monge-Kan\-to\-ro\-vich problem
is to find a cost-minimizing transport plan. More precisely, 
writing $\cpl(\mu,\nu)$  for the set of all couplings (namely, measures) on
$X\times Y$ with $X$-marginal $\mu$ and $Y$-marginal $\nu$, 
the problem is to find 
\begin{equation}\label{G1}\tag{OT}
\textstyle
\inf \left\{ \int c\, \de\P:\P\in \cpl(\mu,\nu) \right\} \, 
\end{equation}
and to identify an optimal transport plan $\P^* \in \cpl(\mu,\nu)$.

The notion of  \emph{$c$-cyclical monotonicity} leads to a \emph{geometric characterization} of optimal couplings. 
Its relevance for \eqref{G1} has been highlighted by Gangbo and McCann \cite{GaMc96}, following earlier works of Knott
and Smith \cite{KnSm92} and R\"uschendorf \cite{Ru96} among others. 

We give here a slightly non-standard definition that  is not inherently tied to the transport problem  and serves our exposition more directly\footnote{{The arguments in \cite[Exercise 2.21, p.79]{Vi03} can be used to prove the equivalence with the  more familiar way of stating $c$-cyclical monotonicity of a set $\Gamma$ in the case of $c$ being the quadratic cost:  usually one requires that for any  $(x_{1}, y_{1}), \dots, (x_{n}, y_{n}) \in \Gamma$, $y_{n+1}=y_1$ it holds $
\sum_{i=1}^n c(x_{i}, y_{i}) \leq \sum_{i=1}^n c(x_{i}, y_{i+1})$. The argument when $c$ is general carry over verbatim.}}.

A set $\Gamma \subseteq X \times Y$ is  \emph{$c$-cyclically monotone} if any positive measure $\alpha$ that is finite and supported on finitely many points in $\Gamma$, is a cost-minimizing transport between its marginals. I.e., if $\alpha'$ has the same marginals as $\alpha$, then 
\begin{align*}
\textstyle
\int c \, \de \alpha \leq \int c \, \de \alpha'.
\end{align*}
A transport plan $\gamma$ is called $c$-cyclically monotone if it is concentrated on such a set $\Gamma$, i.e.\ if there is such a $\Gamma$ with $\gamma(\Gamma) = 1$.

The equivalence of
 optimality and $c$-cyclical monotonicity has been established under progressively milder regularity assumption. Based on 
\cite{AmPr03, Pr07, ScTe09,BeGoMaSc08, BiCa10} the following `\emph{Monotonicity Principle}' holds true: 
\begin{theorem}\label{LGTrans}
Let $c:X\times Y \to [0,\infty)$ be measurable and assume that $\P\in \cpl(\mu, \nu)$ is a transport plan with finite cost $\textstyle\int c\, \de\P \in \R_+$. Then $\P$ is optimal if and only if $\P$ is $c$-cyclically monotone.
\end{theorem}
The importance of this result stems from the observation that it is often an elementary  and feasible task to see whether a transport behaves optimally on a  finite number of points. But  this would be a priori of no help for a problem where single points do not carry positive mass.  Theorem \ref{LGTrans}  provides the required remedy to this obstacle as it establishes  the connection to optimality on a ``pointwise'' level.

 \subsection{Recent developments and aims of this article}
 
 More recently, variants of this `monotonicity principle' have been applied in transport problems for finitely or infinitely many marginals \cite{Pa12fm, CoDeDi15, Gr16a, BeGr14, Za14}, the martingale version of the optimal transport problem \cite{BeJu16, NuSt16, BeNuTo16}, stochastic portfolio theory \cite{PaWo16}, the Skorokhod embedding problem \cite{BeCoHu14,GuTaTo15b}, the distribution constrained optimal stopping problem \cite{BeEdElSc16, BeNuSt20} and the weak transport problem \cite{GoJu18,BaBeHuKa20, BaBePa18}.

What all these articles have in common is that the original idea is applied to other infinite-dimensional \emph{linear} optimization problems. In the present note, we advertise the idea that this optimality principle can be useful beyond linear problems and in fact to problems that are not susceptible to a convex duality approach. Given the versatile applicability of the idea in various linear optimization problems, the extension to non-linear problems appears highly promising. 

In Section \ref{FormalSection} we present the principal idea of what kind of structure such a monotonicity principle might take in applications to non-linear optimization problems. While the heuristic derivation in  Section \ref{FormalSection} is based on a purely formal linearization procedure, we rigorously establish this result in Section \ref{SectionIntermediate} for a large subclass of non-linear problems. We then further specify this rigorous  monotonicity principle in the setup of a general and not necessarily convex version of the Schr\"odinger problem: In Theorem \ref{StructureOfOptimizers1stCase} we show how this non-linear monotonicity principle can be used to obtain necessary optimality conditions, which are shown to be also sufficient for convex problems such as the classical Schr\"odinger problem, see Theorem \ref{converse}. Furthermore we derive novel variants of these conditions for more general entropy functionals in Theorem \ref{StructureOfOptimizers2ndCase}. 

To illustrate the potential of our approach, we apply our results to obtain a shape theorem for the optimal solutions of a non-convex Schr\"odinger problem with congestion. { Furthermore, we discuss briefly how a natural generalization of our findings, which we plan to address in future works, would allow to advance considerably the understanding of the recently introduced mean field Schr\"odinger problem \cite{backhoff2020mean}.} 


\subsection{A `formal' non-linear monotonicity principle}\label{FormalSection}\label{SchroedingerSection}
In this section we introduce some notation and then state a non-linear monotonicity principle which is `formal' in the sense that we do not give a rigorous proof or precise conditions under which it is expected to hold. In the next section we will then provide a rigorous version which is applicable to the Schr\"odinger problem and similar energy minimization problems. 

Let $\Omega$ be a Polish space with $\cal B$ its Borel sigma-algebra. Consider $\F$ a family of real-valued functions on $\Omega$. We suppose either of the following:

\begin{enumerate}
\item $\F$ is a subset of $C_b(\Omega) $, the space of continuous bounded functions.
\item $\F$ is a countable sub-family of $B_b(\Omega)$, the space of Borel bounded functions.
\end{enumerate}

We are given a functional $$G:{\cal P}(\Omega)\to [0,+\infty],$$
and we are interested in the following problem
\begin{align}
\label{P measures}\tag{P}
\inf\left\{ G(\QQ)\,:\,\QQ\in\mathcal{P}(\Omega),\, \QQ\in \mathrm{Adm}\right \},
\end{align}
where
$$\textstyle\mathrm{Adm}:=\mathrm{Adm}_\F := \left\{\QQ:\int f\de\QQ =0,\,\forall f \in\F\right\}, $$
{and $\mathcal{P}(\Omega)$ denotes the set of Borel probability measures on $\Omega$. }

The standing assumption on $G$ is that 
there exist directional derivatives with representation via functions, i.e. for any $\QQ$ in the domain $D(G)=\{\QQ\in \mathcal P(\Omega): G(\QQ)< \infty\}$ there exists $\delta G_{\QQ}:\Omega\to (-\infty, \infty]$ measurable such that
$$\textstyle\forall \bar\QQ\in D(G), \quad\lim_{\epsilon\searrow 0}\frac{G(\QQ+\epsilon[\bar{\QQ}-\QQ])-G(\QQ)}{\epsilon} = \int_\Omega \delta G_{\QQ}(\omega)[\bar{\QQ}-\QQ](\de\omega),$$
where one implicitly assumed the limit to exist for all $\QQ, \bar{\QQ}\in D(G)$.



Positive finite measures $\alpha,\alpha'$ with equal mass and finite support are called \emph{competitors} if  $$\textstyle \int f\,\de(\alpha-\alpha')=0,\,\forall f\in\F.$$ 
We then expect the following:

\begin{pseudotheorem}[Non-Linear Monotonicity Principle, formal version]\label{Thm monotonicity}
Suppose $\QQ^*\in \mathrm{Adm}\cap  D(G)$ is an optimizer for Problem \eqref{P measures}. Then
\begin{enumerate}
\item $\QQ^*$ is a minimum of the \emph{linearized} problem
$$\textstyle\inf\left\{ \int_\Omega c(\omega)\, \QQ(\de\omega)\,:\,\QQ\in\mathrm{Adm} \cap D(G)\right\}, \text{ where } c(\omega) :=\delta G_{\QQ^*}(\omega), $$
\item  
There exists a Borel set $\Gamma_{\Q^*} \subset \Omega$  such that $\Q^*(\Gamma_{\Q^*})=1$ having the following property: given competitors $\alpha, \bar \alpha$, with $\supp \alpha\subseteq \Gamma_{\Q^*}$ we have
$$\textstyle \int\delta G_{\Q^*}\de\alpha\leq \int\delta G_{\QQ^*}\de\bar\alpha. $$
\end{enumerate}
\end{pseudotheorem}

\begin{proof}[Formal derivation]
By optimality of $\QQ^*$ and the fact that $\mathrm{Adm}$ is convex, we easily obtain $$\textstyle\lim_{\epsilon\searrow 0}\frac{G(\QQ^*+\epsilon[\bar{\QQ}-\QQ^*])-G(\QQ^*)}{\epsilon} = \int_\Omega\delta G_{\QQ^*}(\omega)[\bar{\QQ}-\QQ^*](\de\omega)\geq 0,$$
for all $\bar{\QQ}\in \mathrm{Adm}\cap  D(G)$, showing that $\QQ^*$ is a minimum of
the linearized problem in (1). The  monotonicity principle in \cite[Theorem 1.4]{BeGr14} applies, and we find exactly the desired condition in (2).
\end{proof}

\subsection{A Rigorous Non-Linear Monotonicity Principle }
\label{SectionIntermediate}

We consider throughout a continuous function $h:\R_+\to\R_+$ satisfying at least:
\begin{align}\label{eq:assumpt_h}\tag{H}
h\text{ is differentiable on $(0,\infty)$ and the limit $h'(0):=\lim_{x\searrow 0}h'(x)$ exists.}
\end{align}
Throughout we fix $\PP\in\mathcal{P}(\Omega)$ and consider
$$G(\QQ) \,:=G_h(\QQ):= \,\left \{
\begin{array}{ll}\textstyle
\int_\Omega h\left (  \frac{\de\QQ}{\de\PP}(\omega) \right )\PP(\de\omega) &, \text{ if }{\mathcal P(\Omega)\ni}\QQ\ll\PP,\\
+\infty&, \text{otherwise}.
\end{array}
\right. 
$$
and the associated minimization problem
\begin{align}
\label{Ph measures}\tag{P${}_h$}
\inf\left\{ G_h(\QQ)\,:\,\QQ\in\mathcal{P}(\Omega),\, \QQ\in \mathrm{Adm}\right \}.
\end{align}

\begin{lemma}[Non-Linear Monotonicity Principle]\label{SpecialMP}
{In addition to \eqref{eq:assumpt_h},} suppose that $h$ is twice differentiable on $\R_+$ with $h''\geq C$ everywhere for some $C\in \R$ and that  $\lim_{x\rightarrow +\infty} h'(x)=+\infty$. Furthermore, assume that either $h'$ is lower bounded or 
 $\lim_{x\downarrow 0}h'(x)=-\infty$ and let $\QQ^*$ be an optimizer of Problem \eqref{Ph measures}. Then there exist sets $\Gamma_{\QQ^*}, \Gamma_\PP$ such that $\PP(\Gamma_\PP)=\QQ^*(\Gamma_{\QQ^*})=1$ and for all competitors $\alpha, \alpha'$ with $\supp(\alpha)\subseteq \Gamma_{\QQ^*}, \supp(\alpha')\subseteq \Gamma_\PP$ we have
\begin{align}\label{MonotonicityCond}
   \textstyle \int h'\Big(\frac{\de\QQ^*}{\de\PP}\Big)\, \de\alpha\leq \int h'\Big(\frac{\de\QQ^*}{\de\PP}\Big)\, \de\alpha'.
\end{align}
\end{lemma}

{
We defer the proof of the above lemma to Section \ref{sec:proof_lem_SpecialMP}. Typical examples of $h$ satisfying the above conditions are $h(x)=x\log x - x + 1$ or $h(x)=x^p$ with $p>1$.}

\subsection{Schr\"odinger-type Problems}
\label{SchroedingerSection}
We specify the setting of Section \ref{SectionIntermediate}. In this part we are interested in the case $$\Omega:={X}\times {Y},$$
for ${X}, {Y}$ Polish spaces. As for the constraints set $\F$, we are interested in 
$${
\textstyle \F_\mu:=\{\bar f(x,y)=f(x)-\int_{{X}}f\, \de\mu:\,f \in C_b({X})\} \quad \F_\nu:=\{\bar g(x,y)=g(y)-\int_{{Y}}g\, \de\nu:\,g \in C_b({Y})\},}$$
and 
$$\F:=\F_{\mu,\nu}:=\F_\mu \cup \F_\nu,$$
for given probability measures
$\mu\in\mathcal{P}({X}), \nu\in\mathcal{P}({Y}),$  satisfying
$$\mu\ll \text{proj}^\mathcal{X}(\PP) \text{    and   } \nu\ll \text{proj}^\mathcal{Y}(\PP). $$
With these specifications, our minimization problem (Problem \eqref{Ph measures}) clearly becomes:
\begin{align}\label{P regularized transport}\textstyle
\inf\left\{\int_{{X}\times {Y}} h\left (  \frac{\de\QQ}{\de\PP}(x,y) \right )\PP(\de x,\de y): \QQ\in \cpl(\mu, \nu)  \right\}.
\end{align}

 Notice that for the choice $h(x)= x \log(x)$, Problem \eqref{Ph measures} becomes the classical Schr\"odinger problem\footnote{See L\'eonard's  survey \cite{Le14} on classical results around the Schr\"odinger problem and its probabilistic meaning. Recently this problem has seen a surge in interest owing to the overture to machine learning by Cuturi \cite{Cu13}.}.
   


We now rigorously derive necessary optimality conditions for Problem \eqref{P regularized transport}. The functions  $\phi$ and $\psi$ appearing in Theorem \ref{StructureOfOptimizers1stCase} can formally be seen as Lagrange multipliers and in the case $h(x)=x\log(x)$ they are known as Schr\"odinger potentials, see \cite[Sec 2.]{Le14}. {We remind the reader that $\rho\sim\eta$ stands for equivalence of measures in the sense that $\rho\ll\eta$ and $\eta\ll\rho$.}

\begin{theorem}\label{StructureOfOptimizers1stCase}
Assume that Problem \eqref{P regularized transport} is finite and $\QQ^*$ is an optimizer thereof. Importantly we also assume that  $\PP \sim \mu\otimes \nu$. Let $h:[0,\infty) \to (-\infty, \infty)$ be  twice continuously differentiable, $\lim_{x\to0}h'(x)=-\infty$, $\lim_{x\to + \infty} h'(x)=+\infty$ and $\inf_{\R_+} h''>-\infty$. Then $\QQ^*\sim \PP$ and there exist  measurable functions $\phi:{X}\to [-\infty,+\infty)$ and $\psi:{Y}\to [-\infty,+\infty)$ such that
\begin{equation}\label{eq:Kantorovich_potentials}
  \textstyle  h'\circ \frac{\de\QQ^*}{\de\PP}(x,y) =  \phi(x)+\psi(y) ,\,\,\, \PP-a.s.
\end{equation}
\end{theorem}
It is worth remarking that the above theorem applies to $h(x)=x \log (x)$ (where $h'(x)=1+\log(x)$) but not to $h(x)=x^2$ (where $h'(x)=2x$). This latter case (and similar ones) is covered by the following complementary theorem:

\begin{theorem}\label{StructureOfOptimizers2ndCase}
Assume that Problem \eqref{P regularized transport} is finite and $\QQ^*$ is an optimizer thereof. Assume that  $\PP \sim \mu\otimes \nu$.
Let $h:[0,\infty) \to (-\infty, \infty)$ be strictly increasing, continuously differentiable, $\lim_{x\to 0}h'(x)=0$, $\lim_{x\to +\infty} h'(x)=+\infty$, $\inf_{\mathbb R_+}  h''>-\infty$. Then there exist  measurable functions $\phi:{X}\to [-\infty,+\infty)$ and $\psi:{Y}\to [-\infty,+\infty)$ such that
\begin{align}\textstyle h'\circ\frac{\de\QQ^*}{\de\PP}(x,y) =  (\phi(x)+\psi(y))_+ ,\,\,\, \PP-\mbox{a.s.} \label{eq:Kantorovich_potentials2}\end{align}
\end{theorem}

{
We remark that uniqueness of an optimizer to Problem \eqref{P regularized transport} is guaranteed if $h$ is strictly convex. On the other hand, Conditions \eqref{eq:Kantorovich_potentials}-\eqref{eq:Kantorovich_potentials2} do not characterize optimizers even when these are unique (e.g.\ when $h'$ is not one-to-one).}

\subsection*{Comparison with the existing literature}
Minimization problems of the form \eqref{P regularized transport} have been studied for a long time, the most notable example being the Schr\"odinger problem. Indeed, analogues of Theorem \ref{StructureOfOptimizers1stCase} for the case where $h(x)=x \log x$ have been obtained in seminal works of Fortet and Beurling \cite{fortet1940resolution,beurling1960automorphism}. In more recent works, Borwein and Lewis \cite{borwein1992decomposition} and Borwein, Lewis and Nussbaum \cite{borwein1994entropy} proposed an approach to entropy minimization that combines fixed point-arguments and convex optimization techniques. We refer to Gigli and Tamanini's article \cite{gigli2018second} for adaptations of these results to the setting of $RCD$ spaces. Convex duality is also at the heart of the proof strategy of Pennanen and Perkki\"o \cite{pennanen2019convex}. A different viewpoint is adopted by R\"uschendorf and Thomsen \cite{ruschendorf1993note}: therein the shape of the optimal measure is found as a consequence of the closedness property of sum spaces of integrable functions. We also refer to Carlier and Laborde \cite{carlier2020differential} for multidimensional generalizations. A large part of the above mentioned results is surveyed by L\'eonard in \cite{Le14}. This author  has also proven shape theorems for the Schr\"odinger problem analogous to Theorem \ref{StructureOfOptimizers1stCase} in \cite{leonard2001minimizers,leonard2010entropic}. Cattiaux and Gamboa \cite{cattiaux1999large} treat the more general case when $h$ is the log-Laplace transform of a probability measure: this condition implies that $h$ is convex. However, it is not assumed there (unlike what we do here) that $\mathbb{P}\sim\mu\otimes\nu$, but only $\mathbb{P}\ll\mu\otimes\nu$ is needed. Their proofs rely essentially on ideas and tools coming from large deviations and on  the earlier findings of \cite{ruschendorf1993note}. To the best of our knowledge, the case when $h$ is not convex has not been treated before the present article. {As for Lemma \ref{SpecialMP}, a more explicit version in the particular case of the classical Schr\"odinger Problem has been obtained in parallel by Bernton, Ghosal and Nutz in \cite{BeGhNu21}, where it is furthermore leveraged to obtain stability and large deviations estimates.}

\quad

We now study the converse direction: how structure of a measure implies optimality. Here we do need to assume convexity.

\begin{theorem}\label{converse}
Let $h:[0,\infty) \to (-\infty, \infty)$ be strictly convex, lower-bounded, and continuously differentiable, $\lim_{x\to 0}h'(x)=0$, { $\lim_{x\to +\infty}h'(x)=+\infty$, and $h(2x)\leq ah(x)+bx+c $ for constants $a,b,c$}. Suppose that $\QQ^*\in\cpl(\mu, \nu)$ is absolutely continuous with respect to $\PP$, with
\begin{align*}\textstyle h'\circ\frac{\de\QQ^*}{\de\PP}(x,y) =  (\phi(x)+\psi(y))_+ ,\,\,\, \PP-\mbox{a.s.}\end{align*}
for measurable $\phi:{X}\to [-\infty,+\infty)$ and $\psi:{Y}\to [-\infty,+\infty)$. Then $\QQ^*$ is optimal for \eqref{P regularized transport}. 	
\end{theorem}
With the same techniques used to prove Theorem \ref{converse}, variants of this result can be established if $h'(0)\in [-\infty, \infty)$. {This covers in particular the 
Schr\"odinger problem, and relatives thereof, for which the converse direction is contained in Theorem \ref{StructureOfOptimizers1stCase}. As a side remark, we also want to stress that Theorem \ref{StructureOfOptimizers2ndCase} can be plainly }adapted to cover the case $h'(0)\in (-\infty, \infty) $.

\subsection*{A Toy Example: Schr\"odinger Problem with Congestion} 
Let us say $x\in X$ and $y\in Y$ denote respectively origins and destinations for car users in a city. Hence an origin-destination pair $(x,y)$ can stand for the route that a car has to travel from $x$ to $y$. Experts have determined that $\P\in\mathcal P( X\times Y)$ is the optimal use of the road network (here $\P(\de x,\de y)$ is the infinitesimal proportion of cars taking route $(x,y)$) in the stationary case. However, the actual proportion of car trip origins and car trip destinations are described by $\mu\in\mathcal P( X)$ and $\nu\in\mathcal P( Y)$ respectively, rather than $\text{proj}^{ X}(\P)$ and $\text{proj}^{ Y}(\P)$. In the vanilla version of the Schr\"odinger Problem we aim to determine a minimizer $\Q^*$ of the relative entropy $\textstyle\int  \frac{\de \Q}{\de \P}\log\Big( \frac{\de \Q}{\de \P}\Big)\de \P$ over $\Q\in\cpl(\mu,\nu)$, $\Q\ll\P$, amounting to the distribution of car trips compatible with the experts' guess $\P$ and the marginal information $\mu$ and $\nu$. However, we may also want to consider congestion effects, codified by an added term $f\Big( \frac{\de \Q}{\de \P}\Big)$ with $f(\cdot)$ increasing, the idea being that adding traffic above the experts' recommendation should be more costly than the opposite. This way we arrive at the non-convex Schr\"odinger-type problem of minimizing $\textstyle \int \Big [ \frac{\de \Q}{\de \P}\log\Big( \frac{\de \Q}{\de \P}\Big)+ f\Big( \frac{\de \Q}{\de \P}\Big) \Big ]\de \P$ under the same constraints. The optimality condition in Theorem \ref{StructureOfOptimizers1stCase} now reads:
$$\textstyle(\log+f')\Big( \frac{\de \Q^*}{\de \P}\Big)=\phi(x)+\psi(y),$$
from which $\Q^*$ can even be determined depending on the choice of $f$.

{\subsection*{Some perspectives on the mean field Schr\"odinger problem}

In the recent article \cite{backhoff2020mean} a mean field version of the Schr\"odinger problem has been introduced. A simplified discrete-time version of it consists in finding the most likely evolution conditionally to observations at initial and terminal times of the particle system $(X^i_{t})_{i=1,\ldots,N;\, t=0,1,2}$ where $(X^1_0,\ldots,X^n_0)$ are i.i.d. samples from a probability measure $\mu$ on $\mathbb{R}^d$ and

 \begin{equation}\label{toyMFSP-part sys}
 X^i_{t+1}-X^i_t = - \sum_{j\leq N} \nabla W(X^i_t-X^j_t) +  \xi^i_t, \quad i=1,\ldots,N, \,\, t=0,1. 
 \end{equation}
  Here the random variables $(\xi^i_t)_{i=1,\ldots,N; t=1,2}$ are i.i.d.\ standard Gaussians. The large deviations rate function for the empirical distribution of the particle system \eqref{toyMFSP-part sys} in the regime $N\rightarrow +\infty$ is known explicitly (see \cite{DaGae87} for a general result in continuous time and \cite{fischer2014form} for the analysis of the toy model \eqref{toyMFSP-part sys}) and leads to the following problem formulation
\begin{equation}\label{toyMFSP}
    \inf\left\{\int h\left (  \frac{\de\QQ}{\de R(\QQ)}(x_0,x_1,x_2) \right )\, R(\QQ)(\de x_0,\de x_1,\de x_2): \QQ\in \cpl(\mu, \nu)  \right\}.
\end{equation}

In the above we denoted $h(x)=x\log x$ and, adapting the convention used throughout this paper, we denoted by $\cpl(\mu, \nu)$ the subset of $\mathcal{P}(\mathbb{R}^d\times\mathbb{R}^d\times\mathbb{R}^d)$ whose first marginal $(t=0)$ is $\mu$ and whose last $(t=2)$ marginal is $\nu$. Finally, for a given $\QQ$, $R(\QQ)\in\mathcal{P}(\mathbb{R}^d\times\mathbb{R}^d\times\mathbb{R}^d)$ is defined as the law of the controlled discrete stochastic differential equation

\begin{equation}\label{toy-controlledSDE}
\begin{cases}
Z_{t+1} = Z_t -\int \nabla W (Z_t-x_t) \QQ(\de x_0,\de x_1,\de x_2) + \xi_t,\quad t=0,1,\\
Z_0 \sim \mu,
\end{cases}    
\end{equation}
where $(\xi_0,\xi_1,\xi_2)$ are i.i.d.\ standard Gaussians. Despite several analogies with \eqref{P regularized transport}, including the fact that the function $R(\cdot)$ naturally introduces non-convexity into the problem, the analysis of \eqref{toyMFSP} is outside the reach of this work, essentially because the ``reference'' measure $R(\QQ)$ depends on $\QQ$. However, the heurisitcs put forward in the introduction based on the linearization procedure still apply and leads to natural conjectures on the kind of monotonicity principle and shape theorem for optimizers to be expected in this situation. For this reason, the present work is a first step in the direction of developing and exploiting ever more powerful monotonicity principles. One of the main motivations for validating such conjectures for Problem \eqref{toyMFSP} resides in the fact that a shape theorem for the mean field Schr\"odinger problems yields existence of solutions for the coupled Fokker Planck-Hamilton Jacobi Bellman system describing the dynamics of mean field Schr\"odinger bridges. We redirect the interested reader to \cite[Sec 1.3]{backhoff2020mean} for the precise form of such PDE system as well as for more explanations.}

\section{Proofs}

\subsection{Proof of the Non-linear Monotonicity Principle: Lemma \ref{SpecialMP}} \label{sec:proof_lem_SpecialMP}
The proof requires two preliminary results. The first is a lemma telling essentially that, if $G=G_h$ directional derivatives can be computed with
$$\textstyle\delta G_{\QQ}(\omega) = h'\left (  \frac{\de\QQ}{\de\PP}\right )(\omega) .$$
More precisely, we will need this in the form of the following lemma: 
 \begin{lemma}\label{GradToSign} Let $h$ satisfy the hypotheses of Lemma \ref{SpecialMP}. Consider now a probability measure $\Q$ and positive measures $\theta, \theta' $ satisfying
\begin{enumerate}[(i)]
\item $\theta(\Omega)= \theta'(\Omega)$.
\item $\textstyle\int h\left(\frac{\de\Q}{\de\P}\right)\, d\P$ exists and is finite.
\item $\theta \leq \Q, \theta' \leq \P$.
\item There is a constant $l\in \R$ such that $-l\leq h'\left(\frac{\de\Q}{\de\P}\right)\leq l$ hold $\theta + \theta'$-a.s.
\item $\textstyle\int h'\left(\frac{\de\Q}{\de\P}\right)\, \de(\theta'-\theta) < 0$.
\end{enumerate}
Setting $\QQ_\epsilon:= \Q+\epsilon(\theta'-\theta)$ we then find that, for all $\epsilon > 0$ small enough, $\int h\left(\frac{\de\QQ_\epsilon}{\de\P}\right)\, \de\P$ exists and $$\textstyle\int h\left(\frac{\de\QQ_\epsilon}{\de\P}\right)\, \de\P < \int h\left(\frac{\de\Q}{\de\P}\right)\, \de\P.$$
\end{lemma}

\begin{proof}
If $0\leq \epsilon\leq 1$ then $\QQ_\epsilon$ is {by (i) and (iii)} a probability measure. By hypothesis {$h''\geq C$ we have}
\begin{equation}\label{eq:GradtoSign1}\textstyle h\left(\frac{\de\QQ_\epsilon}{\de\P}\right) \geq h\left(\frac{\de\Q}{\de\P}\right)+\epsilon h'\left(\frac{\de\Q}{\de\P}\right)\frac{\de(\theta'-\theta)}{\de\P}- \frac{\epsilon^2 C}{2}\left(\frac{\de(\theta'-\theta)}{\de\P}\right)^2,
\end{equation}
Combining (iii) and (iv)  we get

\begin{align*}
 \textstyle \sup_{ \mathcal{X}\times\mathcal{Y}}\left | \frac{\de(\theta'-\theta)}{\de\P} \right |&\leq \textstyle\sup_{\supp(\theta) \cup\supp(\theta') }   \frac{\de\theta'}{\de\P}+ \frac{\de\theta}{\de\P} \\
  &\leq \textstyle 1+\sup_{ \supp(\theta)}  \frac{\de\QQ}{\de\P} \leq 1+\sup \, (h')^{-1}([0,l])<+\infty,
\end{align*}
where to obtain the last inequality we used that $ \lim_{x\rightarrow +\infty} h'(x)=+\infty$. Using this result in \eqref{eq:GradtoSign1} shows that $\textstyle\int h\left(\frac{\de\QQ_\epsilon}{\de\P}\right)\, \de\P$ exists and belongs to $(-\infty,+\infty]$. Similarly,
\begin{align}\label{eq_conv_ineq_eps}\textstyle
h\left(\frac{\de\Q}{\de\P}\right) - h\left(\frac{\de\QQ_\epsilon}{\de\P}\right)\geq \epsilon h'\left(\frac{\de\QQ_\epsilon}{\de\P}\right)\frac{\de(\theta-\theta')}{\de\P} - \frac{\epsilon^2 C}{2}\left(\frac{\de(\theta'-\theta)}{\de\P}\right)^2.
\end{align}

Next, we observe that if we can prove that for $\gamma=\theta,\theta'$ we have
$$\textstyle\int h'\left(\frac{\de\QQ_\epsilon}{\de\P}\right)\de\gamma\to \int h'\left(\frac{\de\Q}{\de\P}\right)\de\gamma ,$$
then we obtain the conclusion dividing by $\epsilon$ on both sides in \eqref{eq_conv_ineq_eps}, integrating in $\de\PP$ and letting $\epsilon\rightarrow 0$. We only argue in the case when $\lim_{x\downarrow 0}h'(x)=-\infty$, the other case being simpler. In this case, condition (iv) implies that $\gamma$-a.s. $\de\Q/\de\P$ takes values in a compact set of $(0,+\infty)$. Using this last observation and (iii) we deduce that $\gamma$-a.s. $\de\Q_\epsilon/\de\P$, viewed as a function of $x,y$ and $\varepsilon$, takes its values in a compact set of $(0,+\infty)$ provided $\varepsilon$ is small enough. The desired conclusion follows by dominated convergence. 
\end{proof}
The second ingredient, towards the proof of Lemma \ref{SpecialMP}, is the following result from \cite{BeGoMaSc08},  which is a  consequence of a duality result by Kellerer \cite{Ke84}. {We recall that if $\alpha,\beta$ are two measures, we write $\alpha\leq \beta$ if $\alpha(A)\leq \beta(A)$ for all $A$ measurable sets. In the following, we denote by $p_i$ the projection onto the $i$-th coordinate of a product space, so that if $\eta$ is a measure on such product then $p_i(\eta)$ denotes its $i$-th marginal.}

\begin{lemma}[{\cite[Proposition 2.1]{BeGoMaSc08}}]\label{KellLemma} Let $(E_i, m_i), i \leq k$ be Polish probability spaces, and $M$ an  analytic\footnote{\cite[Proposition 2.1]{BeGoMaSc08} is stated only for Borel sets, but the same proof applies in the case where $M$ is analytic.  } subset of $E_1 \times \ldots \times E_k$, then one of the following holds true:\begin{enumerate}
\item[(i)] there exist $m_i$-null sets $M_{i}\subseteq E_i$ such that $M \subseteq \bigcup_{i=1}^k p_{i}^{-1} (M_{i})$, or
\item[(ii)] there is a measure $\eta $ on $E_1 \times \ldots \times E_k$ such that $\eta (M)>0$ and $p_{i}(\eta) \leq m_i$ for $i=1, \dots, k$.
\end{enumerate}
\end{lemma}

All in all, we can prove Lemma \ref{SpecialMP} now:
\begin{proof}[Proof of Lemma \ref{SpecialMP}.]
Set $d:=\frac{d\Q^*}{d\P}$ and $c:=h'\circ d$.
  
We want to find finitely minimal sets $\Gamma_{\QQ^*}, \Gamma_\P$ supporting $\Q^*, \P$. To obtain this, it is sufficient to show that for each $l \in \NN$ there are sets $\Gamma_{\QQ^*}, \Gamma_\P $ of full $\Q^*$ / $\P$ measure such that:  for any finite measure $\alpha$ concentrated on at most $l$ points in $\Gamma_{\QQ^*}$ and satisfying $\alpha(\Omega) \leq 1$ as well as $ c_{\upharpoonright\supp \alpha} \leq l$, there  is no $c$-better competitor $\alpha'$ on at most $l$ points in $\Gamma_\P$ and satisfying  $ c_{\upharpoonright\supp \alpha'} \leq l$.  If we achieve this, we can just take the intersection over countably many such $\Gamma_{\QQ^*}, \Gamma_\P$.

Hence, fix $l$ and define $M$ the subset of $\Omega^l\times \Omega^l$ through
\begin{align*}
M &=  \{ ((z_{1}, \dots, z_{l}),(z_{1}', \dots, z_{l}')) \in \Omega^l \times \Omega^l:\\
 &\textstyle 
 \exists \text{ a  measure }\alpha \text{ on } \Omega, \alpha(\Omega) \leq 1, \text{supp }\alpha\subset\{ z_{1}, \dots, z_{l}\}, -l \leq c_{\upharpoonright\supp \alpha} \leq l , ~
\\  
&\textstyle\text{s.t.\ there is a $c$-better competitor }\alpha',  \text{supp }\alpha'\subset\{ z_{1}', \dots, z_{l}'\}, -l \leq c_{\upharpoonright\supp \alpha'} \leq l \}. 
\end{align*}

Note that $M$ is a projection of the set 
\begin{align*}
\begin{split}
\hat M = & \Big\{ (z_1, \ldots, z_l, \alpha_1, \ldots, \alpha_l, z_1', \ldots, z_l', \alpha_1', \ldots, \alpha_l',) \in \Omega^l \times \R_+^l \times \Omega^l\times \R_+^l:\\
& \textstyle\sum \alpha_i \leq 1, 
 \sum \alpha_i= \sum \alpha_i',  -l\leq c_{\upharpoonright\supp \alpha\cup \supp \alpha'} \leq l {\text{ where }\alpha:=\sum \alpha_i\delta_{z_i},\, \alpha':=\sum \alpha_i'\delta_{z_i'},} \\
& \textstyle\sum \alpha_i f(z_i)=\sum \alpha_i' f(z_i') \text{ for all $f\in \F$ },
 \sum \alpha_i c(z_i)>\sum \alpha_i' c(z_i')
\Big\}.
\end{split}
\end{align*}
The set $\hat M$ is Borel; this is immediate if $\F$ is countable, and otherwise follows from the well-known argument that $\F\subset C_b(\Omega)$ contains a separating sequence. Hence $M$ is an analytic set.

We apply Lemma \ref{KellLemma} to the $l$ copies of the spaces $(\Omega, \Q^*)$, $(\Omega, \P)$ and the set $M$. {To be precise, we take $E_i=\Omega$, $i=1,\dots 2l$, $m_i=\Q^*$ if $i\leq l$ and $m_i=\P$ otherwise. By Lemma \ref{KellLemma},} if (i) holds, then there are sets $N_1, N_2$ with $\Q^*(N_1)=\P(N_2)=0$ such that  $M\subseteq   N_1^l \times \Omega^l\cup \Omega^l\times N_2^l .$ {Indeed, noticing that the set $M$ must be symmetric in its first $l$ coordinates, and also on the remaining $l$ ones, we get that if a point is in $M$, then at least one of its first $l$ coordinates are in a given $\Q^*$-null set $N_1$, or one of the remaining $l$ coordinates are in a given $\P$-null set $N_2$.} We set
$\Gamma_{\QQ^*}:= \Omega \setminus N_1, \Gamma_\P :=  \Omega\setminus N_2$, which have full  $\Q^*$ / $\P$ measure respectively.  From the definition of $M$  it can be directly seen that $\Gamma_{\QQ^*},\Gamma_\P$ are as needed. 

If (i) does not hold, (ii) has to. Hence, let us derive a contradiction from it.

For $j \leq 2, i \leq l$, write $p^j_{i}$ for the projection of an element of $ \Omega^l\times  \Omega^l$ onto its $((j-1)\times l + i)$-th  component.
We may assume that the measure $\eta$ given by Point (ii) in Lemma \ref{KellLemma}  is concentrated on $M$, and also  fulfills $p^1_{i}(\eta) \leq \frac{1}{l} \Q^{*}, p^2_{i}(\eta) \leq \frac{1}{l} \P$ for $i= 1, \dots, l$.

We now apply Jankow -- von Neumann uniformization \cite[Theorem 18.1]{Ke95} to the set $\hat M$ to define a mapping
\begin{align*} 
M &\to \hat M\\
(z_1,\ldots,z_l, z'_1,\ldots,z'_l) &\mapsto \bigl( z_1,\ldots,z_l,\alpha_1(z, z'),\ldots, \alpha_l(z, z'),    
z'_1,\ldots,z'_l,\alpha_1'(z, z'),\ldots, \alpha_l'(z, z') \bigr) 
\end{align*} 
which is measurable with respect to the $\sigma$-algebra generated by the analytic subsets of $ \Omega^l\times  \Omega^l$ in the domain and the Borel $\sigma$-algebra of $\Omega^l \times \R_+^l \times \Omega^l\times \R_+^l$ in the range. In the above, we denoted $z_i$ resp.\ $z_i'$ the i-th coordinate of $z\in \Omega^l$ resp.\ $z'\in \Omega^l$. 
Setting
$$ \textstyle
\alpha_{(z, z')}:=  \sum_{i=1}^l \alpha_i(z, z') \delta_{z_i}, \,
\alpha'_{(z, z')}:=  \sum_{i=1}^l \alpha_i'(z, z') \delta_{z_i'},
$$ 

we thus obtain kernels   $(z,z')\mapsto \alpha_{(z,z')}$, $(z,z')\mapsto \alpha'_{(z,z')}$ from $\Omega^l\times \Omega^l$ with the $\sigma$-algebra  generated by its analytic subsets to $\mathcal P( \Omega)$ with its Borel sets. We use these kernels to define measures $\theta, \theta'$ on the Borel sets of $ \Omega$ through 
\begin{align*}\textstyle
\theta (B) =  \int \alpha_{(z,z')}(B) \, \de\eta (z,z'), \ \theta' (B) =  \int \alpha'_{(z,z')}(B) \, \de\eta (z,z').
\end{align*}
By construction $\theta \leq \Q^{*}$. Indeed we have,  
\begin{equation*}\textstyle
    \theta(B) \leq \sum_{i=1}^l \int \delta_{z_i}(B) \de\eta(z,z') = \sum_{i=1}^lp^1_i(B) \leq \QQ^*(B). 
\end{equation*}
Arguing similarly we obtain $\theta' \leq \P$. Moreover  $\theta' $ is  a $c$-better competitor of $\theta$. To see this, we first observe that for each $f\in \cF$ we have
\begin{align}\label{eq:monot_1}\textstyle
\int_{\Omega} f(\bar z) \; \de \theta'(\bar z) = \int\!\!\!\!\int f(\bar z) \; \de\alpha'_{(z,z')}(\bar z)\de \eta (z,z') = \int\!\!\!\!\int f(\bar z) \; \de\alpha_{(z,z')}(\bar z)\de\eta (z,z') = \int_{\Omega}  f(\bar z) \; \de\theta(\bar z), 
\end{align}
and similarly, since $c \leq l$, $(\theta+\theta')$-a.s.\ we obtain
\begin{align*}\textstyle
\int_{\Omega}  c(\bar z) \; \de \theta'(\bar z) = \int\!\!\!\!\int c(\bar z) \; \de\alpha'_{(z,z')}(\bar z)\de \eta(z,z') < \int\!\!\!\!\int c(\bar z) \; \de\alpha_{(z,z')}(\bar z)\de\eta(z,z') = \int_{\Omega}  c(\bar z) \; \de\theta(\bar z). 
\end{align*}
Therefore, since  $\int c \, \de 
(\theta'-\theta)<0$ we obtain from Lemma \ref{GradToSign} that if we set $\Q^{*}_\epsilon=\Q^{*}+\varepsilon(\theta'-\theta)$, then 
\[ \textstyle\int h\left(\frac{\de\QQ^{*}_\epsilon}{\de\P}\right)\, \de\P < \int h\left(\frac{\de\Q^{*}}{\de\P}\right)\, \de\P\]
for $\varepsilon$ small enough. Since \eqref{eq:monot_1} makes sure that $\Q^*_{\epsilon}\in \mathrm{Adm}$, we have derived a contradiction to the optimality of $\Q^{*}$.
\end{proof}

\subsection{Proof of Necessity: Theorems \ref{StructureOfOptimizers1stCase} and \ref{StructureOfOptimizers2ndCase} }

{In the coming proofs the assumption $\P\sim \mu\otimes \nu$ is used in the following form: we use $\mu\otimes\nu\ll\mathbb P$ to apply \cite[Lemma 4.3]{BeGoMaSc08} and $\mathbb P\ll \mu\otimes\nu$ to guarantee w.l.o.g.\ that, as we trim down certain sets in the product space $X\times Y$, their $X$- and $Y$-projections remain unaffected.  }

\begin{proof}[Proof of Theorem \ref{StructureOfOptimizers1stCase} ]
Let $\Gamma_{\QQ^*}, \Gamma_\PP$ be as in Lemma \ref{SpecialMP}. Passing to subsets if necessary we may assume that $\proj_{X}\Gamma_{\QQ^*}=X, \proj_{Y}\Gamma_{\QQ^*}=Y, \Gamma_{\QQ^*}\subseteq \Gamma_{\P}$. Apparently $\QQ^*\ll \PP$. Hence, shrinking $\Gamma_{\QQ^*}$ by an irrelevant $\P$-null set, we may assume that $\Gamma_{\QQ^*}\subset\{d(x,y)>0\}=\{ h'(d(x,y))>-\infty\}$, with $d:=\de\Q^*/\de\P$. In the present transport case the finitistic optimality property \eqref{MonotonicityCond} boils down to cyclical monotonicity, i.e.\ we find that for $(x_i, y_i) \in \Gamma_{\QQ^*}, {i\leq N}, x_{N+1}=x_0, (x_{i+1}, y_{i})\in \Gamma_\PP$ we have 
\begin{align}\label{Cmon}\textstyle
\sum_{i\leq N} h'\circ d(x_i, y_i) \leq \sum_{i\leq N} h'\circ d(x_{i+1}, y_i).
\end{align}
 We say that $x_i, y_i, i\leq N$ form a $(\Gamma_{\QQ^*}, \Gamma_\PP)$-path if $(x_i, y_i)\in \Gamma_{\QQ^*}$ for $i\leq N$, $(x_{i+1}, y_i)\in \Gamma_\PP$ for $i\leq N-1$. 

Based on the assumption $\P\sim \mu\otimes \nu$  we can apply \cite[Lemma 4.3]{BeGoMaSc08} with the cost function $c:=0$ on ${\Gamma_\PP}$ and $c:=+\infty $ otherwise, to obtain that there exist subsets $\tilde{X}\subset X$ and $\tilde{Y}\subset Y$ {with respectively full measure under $\mu$ and $\nu$, so $\tilde{\Gamma}_{\QQ^*}:=\Gamma_{\QQ^*} \cap (\tilde{X}\times \tilde{Y})$ has $\QQ^*$-full measure, and such that crucially} for any points $(x,y), (\bar x,\bar y)\in \tilde{\Gamma}_{\QQ^*}$ there exists a $(\tilde{\Gamma}_{\QQ^*}, \Gamma_\PP)$-path satisfying $(x_0, y_0)=(x,y)$ and $(x_N, y_N)=(\bar x, \bar y)$. Passing to subsets if necessary, we can w.l.o.g.\ assume that $\tilde{X}=X, \tilde{Y}=Y, \tilde{\Gamma}_{\QQ^*}=\Gamma_{\QQ^*}$. We use this to establish 
\begin{align}\label{dPos}
d(x,y)>0 \quad \mbox{for all} \quad (x,y)\in \Gamma_\PP.
\end{align}
To see this, pick an arbitrary point $(x,\bar y)\in \Gamma_\PP$ and points $\bar x,  y$ such that $(\bar x, \bar y), (x, y)\in \Gamma_{\QQ^*}$ and a $(\Gamma_{\QQ^*}, \Gamma_\PP)$-path $(x_0,y_0):= (x, y), (x_1,y_1), \ldots, (x_N, y_N):= (\bar x, \bar y)$ which connects these points. By \eqref{Cmon} we then have (with $x_{N+1}=x_0$)
\begin{align*}
&\textstyle\sum_{i\leq N} h'\circ d(x_i, y_i) \leq \sum_{i\leq N} h'\circ d(x_{i+1}, y_i)\\
\Leftrightarrow
&\textstyle\sum_{i\leq N} h'\circ d(x_i, y_i) \leq \sum_{i\leq N-1} h'\circ d(x_{i+1}, y_i) + h'\circ d(x, \bar y). 
\end{align*} Since the left-hand side is finitely valued and $h'\circ d (x_{i+1},y_i) < \infty$ for $i\leq N-1$ we obtain indeed  $h'\circ d (x,\bar y)>-\infty$. This establishes \eqref{dPos}. It follows that $\PP\sim \QQ^*$ and, by passing to subsets if needed, we can assume without loss of generality that $\Gamma_{\QQ^*}=\Gamma_\PP$. Next, we say that $(x_i, y_i), i\leq N$ form a $\Gamma_{\QQ^*}$-loop if $(x_i, y_i), (x_{i+1}, y_i)\in \Gamma_{\QQ^*}$ for $i\leq N$, where $ x_{N+1}:=x_0$. Note that for any $\Gamma_{\QQ^*}$-loop we have
\begin{align}\label{CeqLoop}\textstyle
\sum_{i\leq N} h'\circ d(x_i, y_i) = \sum_{i\leq N} h'\circ d(x_{i+1}, y_i);
\end{align}
to see this, apply \eqref{Cmon} twice, i.e.\ to the loop in the usual direction as well as to running the loop in the `reverse' direction. By \cite[Prop.\ 1]{Mi06b}, Condition \ref{CeqLoop} is necessary and sufficient to obtain functions $\phi, \psi$ satisfying 
\begin{align}
\label{eq_joint_sum}
h'\circ d(x,y)= \phi(x)+\psi(y),
\end{align}
for all $(x,y)\in \Gamma_{\QQ^*}$.


Fix $x_0\in X$, and observe that \eqref{eq_joint_sum} yields 
$$\textstyle\sum_{i\leq M} h'\circ d(x_{i+1}, y_i)-h'\circ d(x_{i}, y_i) = \phi(x)-\phi(x_0),$$
whenever $(x_i,y_i),(x_{i+1},y_i)\in \Gamma_{\QQ^*}$ for $i\leq M{\in\mathbb N}$ is such that $x_{M+1}=x$. In particular we have
\begin{align}\label{PhiRep} \textstyle
\phi(x)=\inf\left\{  \sum_{i\leq M} h'\circ d(x_{i+1}, y_i)-h'\circ d(x_{i}, y_i) +\phi(x_0) : (x_i,y_i),(x_{i+1},y_i)\in \Gamma_{\QQ^*}\, \text{for }i\leq M,\, x_{M+1}=x  \right\}.
\end{align}

The right-hand side of \eqref{PhiRep} is upper semi-analytic, hence $\phi$ is upper semi-analytic. {Indeed, if $M$ is fixed in the r.h.s.\ of \eqref{PhiRep}, then we would have the partial infimum of a jointly Borel function, which must be upper semi-analytic; this is also the case as we let $M\in\mathbb N$.} Of course \eqref{PhiRep} pertains if we replace the $\inf$ with a $\sup$, hence $\phi$ is also lower semi-analytic. Putting the two together, we find that $\phi$ is Borel. {(For the same reason that a set is Borel iff it and its complement are analytic (Suslin theorem) we have that $\phi$ must be Borel.)}
\end{proof}





\begin{proof}[Proof of Theorem \ref{StructureOfOptimizers2ndCase}]
The start of the proof is the same as the one of Theorem \ref{StructureOfOptimizers1stCase}: Let $\Gamma_{\QQ^*}, \Gamma_\P$ be as in Lemma \ref{SpecialMP}.
Apparently $\QQ^*\ll \PP$ and we may assume that $\Gamma_{\QQ^*}\subseteq \Gamma_\P$.    
Redefining $d:=\frac{\de\Q^*}{\de\P}$ on an irrelevant $\P$-null set we may assume that $d=+\infty$ {exactly} on $(X\times Y) \setminus \Gamma_\P$.
  
  As above, the finitistic optimality property amounts  to  cyclical monotonicity, i.e.\ we find that for $(x_i, y_i) \in \Gamma_{\QQ^*}, {i\leq N}, x_{N+1}=x_0$   we have
\begin{align}\label{Cmonbis}\textstyle
\sum_{i\leq N} h'\circ d(x_i, y_i) \leq \sum_{i\leq N} h'\circ d(x_{i+1}, y_i).
\end{align}
{Note that we do not have to assume $(x_i, y_{i+1})\in \Gamma_\PP$ since $h'\circ d(x_i, y_{i+1})=+\infty$ whenever $(x_i, y_{i+1}) \notin \Gamma_\PP$.} 
Passing to subsets if necessary we may assume that $\proj_{X}\Gamma_{\QQ^*}=X,\, \proj_{Y}\Gamma_{\QQ^*}=Y$.

We say that $x_i, y_i, i\leq N$ form a $(\Gamma, d)$-path if $(x_i, y_i)\in \Gamma$ for $i\leq N$, $d(x_{i+1}, y_i)< \infty$ for $i\leq N-1$. 

Based on the assumption $\P\sim \mu\otimes \nu$  we can apply \cite[Lemma 4.3]{BeGoMaSc08} (with the cost function $c= 0$ on $\Gamma_\PP$ and $+\infty$ else) to obtain the following:

There exist {respectively  full $\mu,\nu$ measure subsets $X_0\subset X, Y_0\subset Y$, so $\Gamma_0:=\Gamma_{\QQ^*} \cap (X_0\times Y_0)$ has $\QQ^*$-full measure,} such that for any points $(x,y), (\bar x,\bar y)\in \Gamma_0$ there exists a $(\Gamma_0, c)$-path satisfying $(x_0, y_0)=(x,y)$ and $(x_N, y_N)=(\bar x, \bar y)$. Of course, we immediately assume w.l.o.g.\ that $X_0=X, Y_0=Y, \Gamma_0=\Gamma_{\QQ^*}$.

In the terms of \cite{BeGoMaSc08} we would say that {$(\Gamma_{\QQ^*},h'\circ d)$} is \emph{connecting}. It then follows from \cite[Proposition 3.2]{BeGoMaSc08} that there exist Borel functions $\phi:X\to [-\infty, \infty), \psi:Y\to [-\infty, \infty)$ such that for all $x\in X, y\in Y$
$$\textstyle \phi(x)+\psi(y)\leq h' \circ d (x,y)$$
with equality holding $\Q^*$-a.s. { Hence we also have \begin{align*}\textstyle h'\circ\frac{\de\QQ^*}{\de\PP}(x,y) =  (\phi(x)+\psi(y))\,{\bf 1}_{\frac{\de\QQ^*}{\de\PP}(x,y)>0} = (\phi(x)+\psi(y))_+\,\, ,\,\,\, \PP-\mbox{a.s.}\end{align*} }
\end{proof}

\subsection{Proof of Sufficiency: Theorem \ref{converse} }


\begin{proof}[Proof of Theorem \ref{converse}]
If Problem \eqref{P regularized transport} has value $+\infty$ then there is nothing to prove. Hence, let $Zd\PP\in \cpl(\mu, \nu) $ with $I:=\int h(Z)\de\PP<\infty$.

Denote $h^*(y)=\sup_{x\geq 0}\{xy-h(x)\}$, and notice that under the assumptions on $h$ we have $h^*(y)=(h')^{-1}(y_+)y-h\circ(h')^{-1}(y_+) $.  

Introduce $\phi_n(x)=(-n)\vee \phi(x)\wedge n$, $\psi_n(x)=(-n)\vee \psi(x)\wedge n$. Clearly (c.f.\ \cite[Lemma 3]{ScTe09}), on $\{\phi+\psi\geq 0\}$ we have $0\leq \phi_n+\psi_n \nearrow \phi+\psi$, while on $\{\phi+\psi\leq 0\}$ we have $0\geq \phi_n+\psi_n \searrow \phi+\psi$. Since
$$h(Z)\geq (\phi_n+\psi_n)Z-h^*(\phi_n+\psi_n),$$
we find 
\begin{align*}
I& \textstyle\geq \int (\phi_n+\psi_n)Z\de\PP-\int h^*(\phi_n+\psi_n) \de\PP \\
&= \textstyle\int \phi_n\de\mu +\int \psi_n\de\nu-\int h^*(\phi_n+\psi_n) \de\PP \\
&= \textstyle\int (\phi_n+\psi_n)\frac{\de\QQ^*}{\de\PP}\de\PP-\int h^*(\phi_n+\psi_n) \de\PP .
\end{align*}
Since $h^*(\cdot)$ is increasing, we have by monotone convergence $$\textstyle\int_{\phi+\psi\geq 0} h^*(\phi_n+\psi_n) \de\PP \to \int_{\phi+\psi\geq 0} h^*(\phi+\psi) \de\PP \text{, and }\int_{\phi+\psi\leq 0} h^*(\phi_n+\psi_n) \de\PP \to \int_{\phi+\psi\leq 0} h^*(\phi+\psi) \de\PP. $$
Since $h^*(\cdot)\geq -h(0)$ we can collect integrals and conclude
$$\textstyle\int h^*(\phi_n+\psi_n) \de\PP \to \int h^*(\phi+\psi) \de\PP , $$
and the right-hand side is $(-\infty,\infty]$-valued. 

By convexity \[ \textstyle h\Big( \frac{\de\QQ^*}{\de\PP}\Big)+\frac{\de\QQ^*}{\de\PP}h'\Big(\frac{\de\QQ^*}{\de\PP}\Big)\leq h\Big(2\frac{\de\QQ^*}{\de\PP}\Big),\] and since by assumption on $h$ the r.h.s.\ is $\PP$-integrable, we deduce that 
\[ \textstyle \int [\phi(x)+\psi(y)]_+ \de\QQ^*(x,y)= \int h'\Big(\frac{\de\QQ^*}{\de\PP}\Big)\de\QQ^*< +\infty.\]

In similar fashion as above, we get 
$$\textstyle \int (\phi_n+\psi_n)\frac{\de\QQ^*}{\de\PP}\de\PP \to \int (\phi+\psi)\frac{\de\QQ^*}{\de\PP}\de\PP,$$
since $\int_{\phi+\psi\geq 0} (\phi+\psi)\frac{\de\QQ^*}{\de\PP}\de\PP<\infty$ 
 and in particular $\int (\phi+\psi)\frac{\de\QQ^*}{\de\PP}\de\PP\in [-\infty,\infty)$. 
 
 Collecting integrals we have found 
\begin{align*}\textstyle
\int (\phi_n+\psi_n)\frac{\de\QQ^*}{\de\PP}\de\PP-\int h^*(\phi_n+\psi_n) \de\PP &\textstyle\to \int (\phi+\psi)\frac{\de\QQ^*}{\de\PP}\de\PP-\int h^*(\phi+\psi) \de\PP\\
& =  \textstyle\int \left [(\phi+\psi)\frac{\de\QQ^*}{\de\PP}- h^*(\phi+\psi)\right ] \de\PP \\
&\textstyle = \int \left [(\phi+\psi)(h')^{-1}\left( [\phi(x)+\psi(y)]_+ \right)- h^*(\phi+\psi)\right ] \de\PP\\
& \textstyle = \int h\circ (h')^{-1}\left( [\phi(x)+\psi(y)]_+ \right)\de\PP\\
&=\textstyle \int h\left(\frac{\de\QQ^*}{\de\PP} \right)\de\PP.
\end{align*}
We conclude that $I\geq \int h(\de\QQ^*/\de\PP)\de\PP$.
\end{proof}

\bibliography{joint_biblio_M}{}
\bibliographystyle{abbrv}

\end{document}